\documentclass[12pt,a4paper,reqno]{amsart}
\usepackage{amsmath}
\usepackage{amsfonts}
\usepackage{amssymb}
\usepackage{amsthm}

\usepackage{graphicx}
\usepackage{fullpage}

\newcommand{\upperRomannumeral}[1]{\uppercase\expandafter{\romannumeral#1}}
\DeclareMathOperator*{\argmin}{arg\,min}

\theoremstyle{plain}

  \newtheorem{proposition}[]{Proposition}
  \newtheorem{lemma}[]{Lemma}
  \newtheorem{theorem}[]{Theorem}

  \newtheorem{remark}[]{Remark}

\theoremstyle{definition}

\title{The difference between a discrete and continuous harmonic measure}
\author{Jianping Jiang }
\email{jjiang@math.arizona.edu}

\author{Tom Kennedy}
\email{tgk@math.ariozna.edu}
\address{Department of Mathematics\\ University of Arizona\\Tucson, AZ 85721 }

\begin{document}

\begin{abstract}
We consider a discrete-time, continuous-state random walk with steps uniformly distributed in a disk of radius of $h$. For a simply connected domain $D$ in the plane, let $\omega_h(0,\cdot;D)$ be the discrete harmonic measure at $0\in D$
associated with this random walk, and $\omega(0,\cdot;D)$ be the (continuous)
harmonic measure at $0$. For domains $D$ with analytic boundary, we
prove there is a bounded continuous function $\sigma_D(z)$
on $\partial D$ such that for
functions $g$ which are in $C^{2+\alpha}(\partial D)$ for some $\alpha>0$
$$
\lim_{h\downarrow 0} \frac{\int_{\partial D} g(\xi) \omega_h(0,|d\xi|;D)
-\int_{\partial D} g(\xi)\omega(0,|d\xi|;D)}{h}
= \int_{\partial D}g(z) \sigma_D(z)  |dz|.
$$
We give an explicit formula for $\sigma_D$ in terms of the conformal
map from $D$ to the unit disc.
The proof relies on some fine approximations of the potential kernel and
Green's function of the random walk by their continuous counterparts,
which may be of independent interest.
\end{abstract}

\maketitle

\section{Introduction}
Let $D$ be a simply connected domain in the complex plane with $z\in D$. Suppose $B_t^z$ is a complex Brownian motion started at $z$, and $\tau_D:=\inf\{t\geq 0: B_t^z\notin D\}$ is the first exit time. The \textit{(continuous) harmonic measure at $z$} is the probability measure on $\partial D$, $\omega(z,\cdot;D)$, defined by
\begin{equation}\label{eqch}
\omega(z,\Gamma;D)=P(B_{\tau_D}^z\in \Gamma),
\end{equation}
where $\Gamma\subseteq \partial D$.
See \cite{GM05} for more information about continuous harmonic measure.

We place a square lattice with mesh size $h>0$ on $D$. Suppose $\tilde{S}_n^{z_h}$ is a simple random walk in $h\mathbb{Z}^2$ started at $z_h$ where $z_h$ is a closest lattice point to $z$, and $T_{D}:=\inf\{n\geq0: \tilde{S}_n^{z_h}\notin D\}$ is the  first exit time. Then the \textit{discrete harmonic measure at $z_h$} is the probability measure on $\partial D$, $\omega_h(z_h,\cdot;D)$, defined by
$$\omega_h(z_h,\Gamma;D)=P(\overline{\tilde{S}_{T_{D}}^{z_h}}\in \Gamma)$$
where $\overline{\tilde{S}_{T_{D}}^{z_h}}$ is a point in $\partial D$ that is closest to $\tilde{S}_{T_{D}}^{z_h}$ and $\Gamma\subseteq\partial D$. It is well-known that $\omega_h(z_h,\cdot;D)$ converges weakly to $\omega(z,\cdot;D)$ as $h\downarrow 0$.
We would like to study the difference $\omega_h - \omega$
as $h \rightarrow 0$. We expect this difference to be of order $h$, and we
would like to compute the limit of $(\omega_h - \omega)/h$ as $h$ approaches $0$.

The simple random walk is not rotationally invariant, but its
limit as the lattice spacing goes to zero is rotationally invariant.
However, there is no reason to expect the limit $\lim_{h\downarrow0}(\omega_h - \omega)/h$ to be rotationally
invariant. Our simulations indicate that this limit depends on how the lattice is oriented
with respect to the domain. We have not been able to rigorously study
this limit for the simple random walk.
In this paper we consider a random walk
that takes place in the continuum and is rotationally invariant even before
we let the lattice spacing $h$ go to zero.
For this model the ratio $(\omega_h - \omega)/h$ is rotationally invariant.

\smallskip
\noindent{\bf The continuous-state random walk:} The walk we study is given by
$S_n=S_0+X_1+X_2+\cdots+X_n$ where $X_i$'s are i.i.d and $X_i$ is uniformly
distributed in the disk of radius $h>0$. We use $P^x$ to denote the conditional distribution of $\{S_n\}_{n\geq0}$, given $S_0=x$; and write $E^x$ for the corresponding expectation.
\smallskip

Let $T_D=\inf\{n\geq0: S_n\notin D\}$. Now the \textit{discrete harmonic measure at $x$} is defined by
\begin{equation}\label{eqdh}
\omega_h(x,\Gamma;D)=P^x(\overline{S_{T_D}}\in \Gamma)
\end{equation}
where $\overline{S_{T_D}}$ is the orthogonal projection of $S_{T_D}$ onto $\partial D$ whenever such orthogonal projection is well-defined (it is if $\partial D$ is analytic and $h$ is small) and $\Gamma\subseteq\partial D$.

Before stating our main result, we need some terminology.
For $z\in \mathbb{C}$,
we denote the imaginary part of $z$ by $\mbox{Im}(z)$.
Let $\mathbb{H}=\{z\in\mathbb{C}:\mbox{Im}(z)>0\}$ be the upper half-plane.
For $z\in\partial D$, we say $\partial D$ is \textit{locally analytic}
at $z$ if there exists a one-to-one analytic function
$f:\mathbb{D}:=\{\xi\in\mathbb{C}:|\xi|<1\}\rightarrow\mathbb{C}$ such that $f(0)=z$ and
$$f(\mathbb{D})\cap D=f(\mathbb{D}\cap\mathbb{H}).$$
We say $\partial D$ is \textit{analytic} if $\partial D$ is locally
analytic at each point of $\partial D$.

Our main result for the continuous-state random walk is the following theorem.

\begin{theorem}\label{thm1}
Suppose $D\subseteq\mathbb{C}$ is a simply connected and bounded Jordan
domain, and $\partial D$ is analytic. Assume $0 \in D$.
Then there is a bounded continuous function $\sigma_D(z)$ on $\partial D$
such that
\begin{equation*}
\lim_{h\downarrow 0}\frac{\int_{\partial D} g(\xi) \omega_h(0,|d\xi|;D)
-\int_{\partial D} g(\xi)\omega(0,|d\xi|;D)}{h}
= \int_{\partial D} g(z) \sigma_D(z) |dz|
\end{equation*}
for all functions $g$ on $\partial D$ which are $C^{2+\alpha}$ with
respect to arc length along the boundary for some $\alpha>0$.
$\omega_h$ is defined in \eqref{eqdh}, and
$\omega$ is defined in \eqref{eqch}.
\end{theorem}

We prove the theorem in two steps. The first step is
the following proposition.

\begin{proposition}\label{propone}
Suppose $D$ satisfies the conditions in Theorem \ref{thm1}.
Let $g$ be a function on $\partial D$ such that the
harmonic function $f$ on $D$ with boundary data $g$ is in
$C^2(\bar{D})$. Then
\begin{equation}\label{eq5}
\lim_{h\downarrow 0} \frac{\int_{\partial D} g(\xi) \omega_h(0,|d\xi|;D)
-\int_{\partial D} g(\xi)\omega(0,|d\xi|;D)}{h}
= K \int_{\partial D} H_D(0,z) \frac{\partial f}{\partial \mathbf{n}_z}(z) \, |dz|,
\end{equation}
where $H_D(0,z)$ is the Poisson kernel (i.e., $\omega(0,|dz|;D)=H_D(0,z)|dz|$) and $\mathbf{n}_z$ is the inward
unit normal at $z$.
The constant $K$ is given by
$$K=\frac{16}{45\pi}+\frac{8}{\pi}\int_0^{\pi/2}(\sin^2\theta
-(\sin^4\theta)/3-\theta\cos\theta\sin\theta)
E^{i\cos\theta}_{h=1}(|\emph{Im}(S_{T_{\mathbb{H}}})|)d\theta,
$$
where $E^{i\cos\theta}_{h=1}$ is the conditional expectation of $\{S_n\}_{n\geq 0}$
with $h=1$ given $S_0=i\cos\theta$, and
$T_\mathbb{H}:=\inf\{k\geq0: S_k\notin \mathbb{H}\}$.
\end{proposition}

Theorem \ref{thm1} follows from Proposition \ref{propone} if we
show that there is a function $\sigma_D(z)$ on the boundary such that
\begin{equation}\label{eqproptwo}
K \int_{\partial D} H_D(0,z) \frac{\partial f}{\partial \mathbf{n}_z}(z) \, |dz|
= \int_{\partial D} g(z) \sigma_D(z) |dz|.
\end{equation}
This follows from Proposition \ref{proptwo}
in Section \ref{densityproof} and a trivial change of variables.
Proposition \ref{proptwo} gives an explicit
formula for $\sigma_D$.

\begin{remark}
Suppose $\gamma(s)$, $0\leq s\leq\mbox{length}(\partial D)$ is the arc
length parametrization of $\partial D$. If $\gamma(s)\in C^{2+\alpha}$ and
$g(\gamma(s))\in C^{2+\alpha}$ for some $\alpha>0$  then corollary
\upperRomannumeral{2}.4.6 of \cite{GM05} implies that $f\in C^2(\bar{D})$.
Note that this implies that if $g$ satisfies the hypothesis in the
theorem, then it satisfies the hypothesis in Proposition \ref{propone}.
\end{remark}

\begin{remark}
Proposition \ref{propone} is proved when each step of the random walk (i.e., $X_i$)
is uniformly distributed in the disk of radius of $h$. Actually,
the same result but with different $K$ holds if: the distribution
of $X_i$ is rotationally invariant, $X_i$ is supported in the disk
of radius of $h$, and the potential kernel for such a random walk has
similar asymptotics as described in Lemma \ref{lempot}(i.e., $a(x)=C_1\ln|x|+C_2+O(|x|^{-2})$ for some constants $C_1$ and $C_2$). The proof for such general $X_i$ is similar to the one of Proposition \ref{propone}.
\end{remark}

\begin{remark}
We believe Theorem \ref{thm1} also holds for piecewise continuous
functions $g$.
\end{remark}

\begin{remark}
The complicated expression for $K$ is what comes out of our proof.
Monte Carlo simulation of this expression gives
$K=0.2647664 \pm 0.0000026$.
We conjecture, but cannot prove, that $K$ is also given by
the much simpler expression
\begin{equation}
K = \lim_{y \rightarrow \infty}
E^{i y}_{h=1}(|\emph{Im}(S_{T_{\mathbb{H}}})|).
\end{equation}
Monte Carlo simulations of these two formulas for $K$ give values that
agree to within about $10^{-7}$.
\end{remark}

The proof of the theorem relies on fine estimates of the potential kernel and
Green's function for the continuous-state random walk.
We prove the scaled Green's function of the continuous-state
random walk is close to the usual continuous Green's function up
to $O(h)$. We give an estimate up to $O(h^{1+\epsilon})$ for some
$\epsilon>0$ when the Green's function of the continuous-state
random walk is evaluated at points near the boundary.


There is a close relationship between harmonic measures and Dirichlet problems. The continuous Dirichlet problem is to find a harmonic function $f$ in $D$ with prescribed boundary values on $\partial D$. More precisely , find $f\in C(\bar{D})\cap C^2(D)$ such that
\begin{equation}\label{eq1}
\left\{
\begin{array}{ll}
\Delta f(z)=\frac{\partial^2 f(z)}{\partial x^2}+\frac{\partial^2 f(z)}{\partial y^2}=0, & z\in D\\
f(z)=g(z), & z\in \partial D.
\end{array}
\right.
\end{equation}
Of course the existence of a solution depends on the smoothness of $\partial D$ and $g$, and uniqueness can be proved using the maximum principle. If $D$ is regular and $g$ is continuous then the solution of the problem \eqref{eq1} can be written as
$$f(z)=E[g(B_{\tau_D}^z)]=\int_{\partial D} g(\xi)\omega(z,|d\xi|;D).$$
See for example Theorem 8.5 of \cite{MP10} for a proof.

The generator, $\Delta_h$, for the continuous-state random walk
is defined by
\begin{equation}\label{eq4}
\Delta_hf(z)=\frac{1}{\pi h^2}\int_{B(0,h)}[f(z+\xi)-f(z)]d\xi,
\end{equation}
where $B(0,h)$ denotes the disk with radius $h$ centered at $0$,
and  $d\xi$ denotes the usual two-dimensional Lebesgue measure.

We divide $D$ into two subdomains, $D_2:=\{z\in D: \mbox{dist}(z,\partial D)\leq h\}$ and $D_1:=D\setminus D_2$ where $\mbox{dist}(z,\partial D)=\inf\{|z-w|: w\in\partial D\}$. Also, we let $D_3:=\{z\in\mathbb{C}\setminus D:\mbox{dist}(z,\partial D)<h\}$. The discrete problem that we consider in this paper is defined by
\begin{equation}\label{eq2.4}
\left\{
\begin{array}{ll}
\Delta_h f_h(z)=0, & z\in D\\
f_h(P)=g(\overline{P}), & P\in D_3
\end{array}
\right.
\end{equation}
where $\overline{P}\in \partial D$ satisfies $|\overline{P}-P|=\min\{|z-P|: z\in\partial D\}$. Note that such $\overline{P}$ is unique when $\partial D$ is analytic and $h$ is small. It is easy to check that $f_h(z)=E^z[g(\overline{S_{T_D}})]=\int_{\partial D}g(\xi)\omega_h(z,|d\xi|;D)$ is a solution of \eqref{eq2.4}, and the uniqueness follows from the maximum principle described in Lemma \ref{lemmax}.

In the PDE literature, $f_h-f$ is usually called the ``discretization error". If one defines \eqref{eq2.4} using the generator for the simple random walk, then $|f_h-f|=o(1)$ was established in \cite{CFL28} while $|f_h-f|=O(h)$
was proved in \cite{Ger30}. See \cite{Was57}, Section 23 of \cite{FW60}, \cite{BHZ68}
and references therein for more about the discretization error.
Since $\int g(\xi) \omega_h(z,|d\xi|;D)=f_h(z)$ and
$\int g(\xi)\omega(z,|d\xi|;D)=f(z)$, our theorem implies that $|f_h-f|=O(h)$.
To our knowledge, all the existing results with discretization error
of $O(h)$ assume that $f\in C^3(\bar{D})$ (see \cite{Ger30} or
Section 23 of \cite{FW60}). Our theorem proves the same discretization error
$O(h)$ but only assuming $f\in C^2(\bar{D})$.

The organization of the paper is as follows. In Section \ref{secpre} we
define the Green's function for the continuous-state random walk and
approximate it by its continuous counterpart. In Section \ref{secmain}
we give finer estimate for the Green's function and prove Proposition \ref{propone}.
In Section \ref{densityproof} we prove \eqref{eqproptwo}.

In the Appendix, we prove the asymptotics for the potential kernel of
the continuous-state random walk.

\section{Preliminaries}\label{secpre}
\subsection{Continuous-state random walk}
Recall that our continuous-state random walk is defined by $S_n=S_0+X_1+X_2+\cdots+X_n$ for $n\in\mathbb{N}\cup \{0\}$, where $X_i$'s are i.i.d and each $X_i$ is uniformly distributed in the disk of radius $h>0$. We use $P^x$ to denote the conditional distribution of $\{S_n\}_{n\geq0}$, given $S_0=x$; and write $E^x$ for the corresponding expectation. For simplicity we suppress the $h$ dependence in $P^x$ and $E^x$ since all random walks have $h$ dependence in this paper. We first define the transition density for such random walks:
\begin{equation}\label{eq2.1}
\begin{array}{ll}
&p(0,x,y)=\delta(y-x),\\
&p(k,x,y)=\lim_{\epsilon\downarrow0}\frac{P^x(|S_k-y|\leq\epsilon)}{\pi\epsilon^2},~k\geq 1,
\end{array}
\end{equation}
where $\delta(y-x)$ for fixed $x$ is the delta function giving unit mass to the point $x$.
For a simply connected and bounded domain $D$, we define the first time the continuous-state random walk leaves $D$ as $T_D$, i.e.,
$$T_D:=\inf\{k\geq0: S_k\notin D\}.$$
We denote the transition density for continuous-state random walk killed on exiting $D$ by $p_D$, i.e., for $x\in D$ and $y\in\mathbb{C}$
\begin{equation}\label{eq2.2}
\begin{array}{ll}
&p_D(0,x,y)=\delta(y-x)\\
&p_D(k,x,y)=\lim_{\epsilon\downarrow0}\frac{P^x(|S_k-y|\leq\epsilon,~k<T_D)}{\pi\epsilon^2},~k\geq 1.
\end{array}
\end{equation}
One immediate consequence of the Markov property for $S_n$ is for $y\in D$
\begin{equation}\label{eq2.3}
p_D(k,x,y)=\frac{1}{\pi h^2}\int_{B(0,h)}p_D(k-1,x,y+\xi)d\xi,~k\geq 1.
\end{equation}
Note that the same equality holds if $p_D$ is replaced by $p$ in \eqref{eq2.3}.
The discrete-time Green's function for $D$ is defined  by
$$G_h(x,y)=\sum_{k=0}^{\infty}p_D(k,x,y),~x\in D,~y\in \mathbb{C}\setminus\{x\}.$$

Recall the definitions of $D_1, D_2$ and $D_3$ in the introduction.
\begin{lemma}\label{lem1}
For a fixed $x\in D$, the discrete-time Green's function satisfies
\begin{equation}\label{eq2.5}
\left\{
\begin{array}{ll}
\Delta_h G_h(x,y)=0, & y\in D\setminus\{x\}\\
G_h(x,y)=0, & y\in D_3.
\end{array}
\right.
\end{equation}
\end{lemma}
\begin{proof}
The boundary condition is obvious. $\Delta_h G_h(x,y)=0$ for $y\in D\setminus\{x\}$ follows from \eqref{eq2.3} and the Fubini-Tonelli theorem.
\end{proof}

The following is a maximum principle for the discrete Laplacian defined by \eqref{eq4}.
\begin{lemma}\label{lemmax}
Let $D$ be a simply connected and bounded Jordan domain. Suppose for some $M\geq0$, $f$ is a function in $D\cup D_3$ satisfying
$$|\Delta_hf(z)|\leq M, z\in D.$$
Then we have
$$\sup_{z\in D}|f(z)|\leq 2Mr^2(D)/h^2+\sup_{z\in D_3}|f(z)|,$$
where $r(D):=\inf\{r>0:D\cup D_3\subseteq B(0,r)\}$ is the radius of the smallest circle (centered at $0$) circumscribed about $D\cup D_3$.
\end{lemma}
\begin{proof}
The proof is similar to the proofs for Lemmas 23.4 and 23.5 in \cite{FW60}.
\end{proof}

Next, we write the discretization error $f_h-f$ in terms of the discrete-time Green's function.
\begin{lemma}\label{lem3}
Under the assumption of Proposition \ref{propone}, we have
\begin{equation}\label{eq2.7}
f_h(0)-f(0)=\int_{D_2} G_h(0,z)\Delta_h f(z) dz.
\end{equation}
\end{lemma}

\begin{proof}
We use a martingale formula (see Lemma 1 of \cite{Kes91}). We extend $f$ to $D\cup D_3$ by setting $f(P)=f(\overline{P})$ for $P\in D_3$ where $\overline{P}=\argmin\{|z-P|: z\in\partial D\}$. It is clear that
$$f(S_n)-\sum_{k=0}^{n-1}\Delta_h f(S_k), n\geq 0$$
is a martingale with respect to $\mathcal{F}_n:=\sigma(S_0,S_1,\cdots,S_n)$ (note that this is actually true for any $f$). Then the optional sampling theorem and the dominated convergence theorem (noting that $f$ is bounded) give
$$E[f(S_{T_D})-f(S_0)]=E[\sum_{k=0}^{ T_D-1}\Delta_h f(S_k)].$$
Suppose $S_0=0$. Then we have
\begin{eqnarray*}
f_h(0)-f(0)&=&E^0[\sum_{k=0}^{ T_D-1}\Delta_h f(S_k)]=E^0[\sum_{k=0}^{\infty}\Delta_hf(S_k)I_{\{k\leq T_D-1\}}]\\
&=&\sum_{k=0}^{\infty}\int_D p_D(k,0,z)\Delta_h f(z)dz\\
&=&\int_D G_h(0,z)\Delta_h f(z)dz=\int_{D_2} G_h(0,z)\Delta_h f(z)dz
\end{eqnarray*}
where we have used Fubini's theorem (since $E^0|\sum_{k=0}^{ T_D-1}\Delta_h f(S_k)|\leq 2 \|f\|_{\infty}E^0 T_D<\infty$) and the fact that $\Delta_hf(z)=0$ for $z\in D_1$ (by the mean value property for $f$).
\end{proof}

In order to estimate the discrete-time Green's function, we need some result concerning the potential kernel. For $x\in\mathbb{C}$, let $a_h(x)$ be the \textit{potential kernel} for the continuous-state random walk defined by
\begin{equation}\label{eq2.8}
a_h(x)=\sum_{k=1}^{\infty}[p(k,0,0)-p(k,0,x)].
\end{equation}
Note that we do not include the $k=0$ term in the above sum so that $a_h(x)$ is a function rather than a distribution. For simplicity, we write $a(x)$ for $a_1(x)$. So $a(x)$ does not depend on $h$.

\begin{lemma}\label{lempot}
The series in \eqref{eq2.8} is convergent, i.e., $a(x)$ is well-defined. $\Delta_{h} a(x/h)=\frac{1}{\pi}I_{B(0,h)}(x)$, and there exists a constant $C_0>0$ such that
$$a(x)=\frac{4}{\pi}\ln|x|+C_0+O(|x|^{-2}) \mbox{ as } x\rightarrow\infty.$$
\end{lemma}

Since the analogous result is standard for discrete-state random walks but the proof is quite long, we will present the proof in the Appendix.

\begin{remark}
See \cite{FU96} for a proof of the above lemma for discrete-state random walks. See also \cite{KS04} for a different proof. \cite{CS11} has a proof for a discrete-state random walk on isoradial graphs.
\end{remark}

\subsection{The difference between discrete-time and continuous Green's functions}
For the domain $D$ defined in Theorem \ref{thm1}, the continuous Green's function $G_D(0,z)$ is the unique harmonic function on $D\setminus\{0\}$ such that $G_D(0,z)\rightarrow 0$ as $z\rightarrow \partial D$ and $G_D(0,z)=-\ln|z|/(2\pi)+O(1)$ as $z\rightarrow 0$. Then $G_D(0,z)=-\ln|z|/(2\pi)+\psi(z)$ where $\psi(z)$ is defined by \eqref{eq2.12} below. Following the estimate of discrete Green's function for simple random walk in \cite{Was57}, we prove the following estimate for the continuous-state random walk.
\begin{lemma}\label{lem4}
Suppose $D\subseteq\mathbb{C}$ is a simply connected and bounded Jordan domain, and $\partial D$ is analytic. Assume $0\in D$. Then there exists a $C>0$ independent of $h$ such that
\begin{equation}\label{eq2.10}
|h^2G_h(0,z)-8G_D(0,z)|\leq Ch
\end{equation}
uniformly for $z\in D$ satisfying $|z|>\sqrt{h}$.
\end{lemma}
\begin{remark}
See Theorem 1.2 of \cite{KL05} and Appendix A of \cite{BJVK13} for results on the difference between the discrete Green's function for the simple random walk and the continuous Green's function for domains without any smoothness assumption on the boundary.
\end{remark}
\begin{proof}
We apply the method introduced in \cite{Was57}. The basic idea of the proof is that we will find the relationship between $G_h(0,z)$ and $G_D(0,z)$ via the potential kernel $a(x)$.

Without loss of generality, we may assume $B(0,\sqrt{h})\subseteq D$.
For $z\in\mathbb{C}$, let $H_h(z)=-h^2\delta(z)-a(z/h)-8\ln h/(2\pi) +C_0$. Then by Lemma \ref{lempot}, we have
$$H_h(z)=-8\ln|z|/(2\pi)+O(h^2/|z|^2) \mbox{ for } z\in D\setminus B(0,\sqrt{h}).$$

For $z\in\mathbb{C}\setminus\{0\}$, let $e_h(z):=h^2G_h(0,z)-H_h(z)$. Note that both terms in the difference contain a delta function at $0$, but these delta functions cancel. We have
$$e_h(z)=h^2\sum_{k=1}^{\infty}p_D(k,0,z)+a(z/h)+8\ln h/(2\pi)-C_0, z\in\mathbb{C}\setminus\{0\}.$$
We can use the equation above to define $e_h(0)$. Then by \eqref{eq2.5}, Lemma \ref{lempot}, and the above discussion, we have
\begin{equation}\label{eq2.11}
\left\{
\begin{array}{ll}
\Delta_h e_h(z)=0, & z\in D\\
e_h(z)=8\ln|z|/(2\pi)+O(h^2/|z|^2), & z\in D_3.
\end{array}
\right.
\end{equation}
Suppose $\psi(z)$ is the harmonic function of $z\in D$ satisfying
\begin{equation}\label{eq2.12}
\left\{
\begin{array}{ll}
\Delta \psi(z)=0, & z\in D\\
\psi(z)=\ln|z|/(2\pi), & z\in \partial D.
\end{array}
\right.
\end{equation}
Note that $\psi(z)=G_D(0,z)+\ln|z|/(2\pi)$ for $z\in D\setminus\{0\}$, and $G_D(0,z)$ can be extended to a harmonic function in a domain containing $\bar{D}\setminus\{0\}$ (see Lemma \upperRomannumeral{2}.2.4 of \cite{GM05}) and $\ln|z|/(2\pi)$ is harmonic in a domain containing $\bar{D}\setminus\{0\}$. Hence $\psi(z)$ can be extended to a harmonic function in a domain containing $\overline{D\cup D_3}$ when $h$ is small. Therefore for all small $h>0$
\begin{equation}\label{eq2.13}
\left\{
\begin{array}{ll}
\Delta_h \psi(z)=0, & z\in D\\
\psi(z)=\ln|z|/(2\pi)+O(h), & z\in D_3
\end{array}
\right.
\end{equation}
where the $O(h)$ comes from the harmonic extension of $\psi(z)$.
Subtracting $8\times$\eqref{eq2.13} from \eqref{eq2.11}, we get (note that we assumed $B(0,\sqrt{h})\subseteq D$ )
\begin{equation}\label{eq2.14}
\left\{
\begin{array}{ll}
\Delta_h [e_h(z)-8\psi(z)]=0 & z\in D\\
e_h(z)-8\psi(z)=O(h), & z\in D_3.
\end{array}
\right.
\end{equation}

Then the maximum principle for $\Delta_h$, Lemma \ref{lemmax},  implies
$$e_h(z)-8\psi(z)=O(h), z\in D.$$
Therefore,
\begin{eqnarray*}
h^2G_h(0,z)&=&e_h(z)+H_h(z)=-8\ln|z|/(2\pi)+8\psi(z)+O(h)\\
&=&8G_D(0,z)+O(h)
\end{eqnarray*}
where the second equality is true if $|z|>\sqrt{h}$.

\end{proof}

\section{Proof of Proposition \ref{propone}  }\label{secmain}
\subsection{Further estimate of the discrete-time Green's function}
Recall that if the boundary of $D$ is analytic, then the Poisson kernel $H_D(0,x)$ for $0\in D, x\in\partial D$ is defined by
$$H_D(0,x)=\frac{\partial G_D(0,x)}{\partial \mathbf{n}_x},$$
where $G_D$ is the continuous Green's function and $\mathbf{n}_x$ is the inward unit normal at $x$. We have the following easy estimate
\begin{lemma}\label{lem5}
Suppose $D\subseteq\mathbb{C}$ is a simply connected and bounded Jordan domain, and $\partial D$ is analytic. Assume $0\in D$. Then for $l\in[0,h]$
$$G_D(0,x+l\mathbf{n}_x)=l H_D(0,x)+O(h^2),$$
uniformly for $x\in\partial D$.
\end{lemma}
\begin{proof}
Note that $G_D(0,x)=0$ for any $x\in\partial D$. Lemma \upperRomannumeral{2}.2.4 of \cite{GM05} implies $G_D$ could be extended to a harmonic function in a domain containing $\bar{D}\setminus\{0\}$. The lemma follows by Taylor expansion of $G_D(0,x+l\mathbf{n}_x)$ about $x$ with coordinate directions $\mathbf{n}_x$ and $-i\mathbf{n}_x$.
\end{proof}

Next, we improve the estimate in Lemma \ref{lem4}. We need a Beurling-type estimate for the continuous-state random walk. For $u\in D$ and $V\subset\partial D$, we define
$$\mbox{dist}_{D}(u;V)=\inf\{R>0: u \mbox{ and } V \mbox{ are connected in }D\cap B(u,R)\}.$$
\begin{lemma}\label{lember}
There exist two absolute constants $\beta>0$ and $C>0$ such that for any simple connected domain $D$, $u\in D$ and $V\subset \partial D$ we have
$$P^u(\overline{S_{T_D}}\in V)\leq C\left[\frac{\emph{dist}(u,\partial D)}{\emph{dist}_{D}(u;V)}\right]^{\beta},$$
where $\overline{S_{T_D}}$ is the orthogonal projection of $S_{T_D}$ onto $\partial D$, i.e., $\overline{S_{T_D}}=\argmin\{|z-S_{T_D}|: z\in\partial D\}$.
\end{lemma}
\begin{proof}
The proof is almost the same as the proof of Proposition 2.11 in \cite{CS11} for a discrete-state random walk.
\end{proof}

The following proposition is an improvement of Lemma \ref{lem4} for points very close to the boundary of the domain.
\begin{proposition}\label{prop}
Suppose $D\subseteq\mathbb{C}$ is a simply connected and bounded Jordan domain, and $\partial D$ is analytic. Assume $0\in D$. Then for any $\epsilon\in (0,1/2)$,  $x\in\partial D$ and $l\in[0,h]$
\begin{equation*}
h^2G_h(0,z)-8G_D(0,z)-8H_D(0,x)E^{li}[|\emph{Im}(S_{T_{\mathbb{H}}})|]=O(h^{1+\epsilon\beta})+O(h^{2-2\epsilon})
\end{equation*}
where $z=x+l\mathbf{n}_x \in D_2$, and the big $O$ terms only depend on $D$.
\end{proposition}
\begin{proof}
From the proof of Lemma \ref{lem4}, we know it is enough to prove for $z=x+l\mathbf{n}_x \in D_2$
\begin{equation}\label{eqgre1}
e_h(z)-8\psi(z)=8H_D(0,x)E^{li}[|\mbox{Im}(S_{T_{\mathbb{H}}})|]+O(h^{1+\epsilon\beta})+O(h^{2-2\epsilon}).
\end{equation}

If we write out the $O(h)$ in \eqref{eq2.14}, we see $e_h(z)-8\psi(z)$ satisfies
\begin{equation*}
\left\{
\begin{array}{ll}
\Delta_h [e_h(z)-8\psi(z)]=0 & z\in D\\
e_h(z)-8\psi(z)=-8G_D(0,z)+O(h^2), & z\in D_3,
\end{array}
\right.
\end{equation*}
where we have used the same notation (i.e., $G_D$) for the harmonic extension of the continuous Green's function. It is easy to see that Lemma \ref{lem5} is also true for $l\in[-h,0]$, i.e., $x+l\mathbf{n}_x\in D_3$. Let $F_h(0,z)$ be the solution of the following discrete Dirichlet problem
\begin{equation*}
\left\{
\begin{array}{ll}
\Delta_h [F_h(0,z)]=0 & z\in D\\
F_h(0,z)=-lH_D(0,x), & z=x+l\mathbf{n}_x\in D_3,
\end{array}
\right.
\end{equation*}

Then \eqref{eqgre1} follows from Lemmas \ref{lemmax} and \ref{lem5}, and the following \textbf{claim}:  for $z=x+l\mathbf{n}_x \in D_2$
\begin{equation}\label{eqgre2}
F_h(0,z)=H_D(0,x)E^{li}[|\mbox{Im}(S_{T_{\mathbb{H}}})|]+O(h^{1+\epsilon\beta})+O(h^{2-2\epsilon}).
\end{equation}
Let $\mathbf{t}_x$ be the unit tangent vector at $x\in\partial D$. One can choose either $\mathbf{t}_x=i\mathbf{n}_x$ or $\mathbf{t}_x=-i\mathbf{n}_x$. Let $\gamma(s)$, $0\leq s\leq |\partial D|$ be an arc-length parametrization of $\partial D$. For any $x\in\partial D$, let $\sigma(x)\in [0,|\partial D|]$ such that $\gamma(\sigma(x))=x$, i.e., $\sigma$ is the inverse of $\gamma$. Since $\partial D$ is analytic,
$$d_H\left(\gamma[\sigma(x)-h^{1-\epsilon},\sigma(x)+h^{1-\epsilon}],\mathbf{t}_x\cdot[-h^{1-\epsilon},h^{1-\epsilon}]\right)\leq M h^{2-2\epsilon}$$
uniformly for $x\in \partial D$, where $d_H$ is the Hausdorff distance and $\mathbf{t}_x\cdot[-h^{1-\epsilon},h^{1-\epsilon}]$ is the tangent line segment centered at $x$ with total length $2h^{1-\epsilon}$. For any $z_1=x_1+l_1\mathbf{n}_{x_1}\in D_3$ where $x_1\in\partial D$, define $x:D_3\rightarrow \partial D$ and $l:D_3\rightarrow [-h,0]$ such that $x(z_1)=x_1$ and $l(z_1)=l_1$. Note that $x(z_1)$ and $l(z_1)$ are uniquely defined if $h$ is small enough. For any $z_0=x_0+l_0\mathbf{n}_{x_0}\in D_2$, the comment after \eqref{eq2.4} implies
$$F_h(0,z_0)=E^{z_0}[-l(S_{T_D})H_D(0,x(S_{T_D}))].$$

Lemma \ref{lember} gives (noting that $x(S_{T_D})=\overline{S_{T_D}}$)
$$P^{x_0+l_0\mathbf{n}_{x_0}}\left(x(S_{T_D})\notin \gamma[\sigma(x_0)-h^{1-\epsilon},\sigma(x_0)+h^{1-\epsilon}]\right)=O(h^{\epsilon\beta}).$$

Therefore,
$$F_h(0,z_0)=E^{z_0}[-l(S_{T_D})H_D(0,x(S_{T_D}))I_{\{\gamma[\sigma(x_0)-h^{1-\epsilon},\sigma(x_0)+h^{1-\epsilon}]\}}(x(S_{T_D}))]+O(h^{1+\epsilon\beta}).$$

The smoothness of $H_D(0,x)$ in $x\in\partial D$ implies $H_D(0,x(S_{T_D}))=H_D(0,x_0)+O(h^{1-\epsilon})$ if $x(S_{T_D})\in \gamma[\sigma(x_0)-h^{1-\epsilon},\sigma(x_0)+h^{1-\epsilon}] $. Hence
$$F_h(0,z_0)=H_D(0,x_0)E^{z_0}[-l(S_{T_D})I_{\{\gamma[\sigma(x_0)-h^{1-\epsilon},\sigma(x_0)+h^{1-\epsilon}]\}}(x(S_{T_D}))]+O(h^{2-\epsilon})+O(h^{1+\epsilon\beta}).$$

Let
$$B_{x_0}(h^{1-\epsilon},M h^{2-2\epsilon}):=\mathbf{t}_{x_0}\cdot[-h^{1-\epsilon}-M h^{2-2\epsilon},h^{1-\epsilon}+M h^{2-2\epsilon}]+\mathbf{n}_{x_0}\cdot[-M h^{2-2\epsilon},M h^{2-2\epsilon}]$$
be the rectangle centered at $x_0$ with length $2h^{1-\epsilon}+2M h^{2-2\epsilon}$ and width $2M h^{2-2\epsilon}$. Then
$$\gamma[\sigma(x_0)-h^{1-\epsilon},\sigma(x_0)+h^{1-\epsilon}]\subseteq B_{x_0}(h^{1-\epsilon},M h^{2-2\epsilon}).$$

It is not hard to see
$$P^{z_0}\left(S(T_D)\in B_{x_0}(h^{1-\epsilon},M h^{2-2\epsilon}) \right)\leq\sup_{0\leq\tilde{l}\leq h}P^{\tilde{l}i}(|\mbox{Im}(S_{T_{\mathbb{H}}})|\leq 2M h^{2-2\epsilon})=O(h^{1-2\epsilon}).$$
Let $B_{x_0}^c(h^{1-\epsilon},M h^{2-2\epsilon}):=\mathbb{C}\setminus B_{x_0}(h^{1-\epsilon},M h^{2-2\epsilon})$. Then we have
\begin{eqnarray*}
F_h(0,z_0)&=&H_D(0,x_0)E^{z_0}[-l(S_{T_D})I_{\{\gamma[\sigma(x_0)-h^{1-\epsilon},\sigma(x_0)+h^{1-\epsilon}]\}}(x(S_{T_D}))\\
&&\times I_{B_{x_0}^c(h^{1-\epsilon},M h^{2-2\epsilon})}(S(T_D))]+O(h^{2-2\epsilon})+O(h^{1+\epsilon\beta})\\
&=&H_D(0,x_0)E^{li}[|\mbox{Im}(S_{T_{\mathbb{H}}})|]+O(h^{2-2\epsilon})+O(h^{1+\epsilon\beta}).
\end{eqnarray*}
This completes the \textbf{claim} (and hence the proposition) since $z_0$ is arbitrary.

\end{proof}

\subsection{A change of variables formula and an estimate of $\Delta_hf$}
From Lemma \ref{lem3}, we know the difference of $f_h(0)$ and $f(0)$ can be represented by a two-dimensional integral in $D_2$. We give a change of variables formula for such an integral in the following lemma.
\begin{lemma}\label{lem6}
Suppose $D$ is a simply connected and bounded Jordan domain, and $\partial D$ is analytic. Let $F$ be a Lebesgue measurable and bounded function on $D_2$. Suppose $z(t)=(u(t),v(t)), 0\leq t\leq 1$ is a parametrization of $\partial D$ such that $u^{\prime 2}(t)+v^{\prime2}(t)\neq 0$ for any $t$. Then the following holds for all small $h>0$
$$\int_{D_2}F(z)dz=\int_0^1\int_0^h F(z(t)+l\mathbf{n}_{z(t)})(1-l\frac{u^{\prime}v^{\prime\prime}-u^{\prime\prime}v^{\prime}}{(u^{\prime2}+v^{\prime2})^{3/2}})dl\sqrt{u^{\prime2}+v^{\prime2}}dt,$$
where $\mathbf{n}_{z(t)}$ is the inward unit normal at $z(t)$.
In particular, for all small $h>0$
$$\int_{D_2}F(z)dz=(1+O(h))\int_{\partial D}\int_0^h F(\xi+l\mathbf{n}_{\xi})dl|d\xi|,$$
where $O(h)$ only depends on $D$.
\end{lemma}
\begin{proof}
The lemma follows from the change of variables
$$z=(x,y)=z(t)+l\mathbf{n}_{z(t)}=(u(t)-\frac{lv^{\prime}}{\sqrt{u^{\prime2}+v^{\prime2}}},v(t)+\frac{lu^{\prime}}{\sqrt{u^{\prime2}+v^{\prime2}}}).$$
Note that one needs to pick $h$ small to make sure the above change of variables is one-to-one.
\end{proof}

We still need some estimate on $\Delta_hf$ for the $f$ defined in Proposition \ref{propone}.
\begin{lemma}\label{lem7}
Under the assumption of Proposition \ref{propone}, for any $x\in\partial D$ and $l\in[0,h]$ we have
$$\Delta_h f(x+l\mathbf{n}_x)=\frac{1}{\pi h^2}\frac{\partial f(x)}{\partial \mathbf{n}_x}[\frac{2}{3}h^2\sqrt{h^2-l^2}+\frac{l^2}{3}\sqrt{h^2-l^2}-lh^2\arccos(l/h)]+O(h^2),$$
where $O(h^2)$ only depends on $f$ and $D$.
\end{lemma}
\begin{proof}
\begin{eqnarray}
\Delta_h f(x+l\mathbf{n}_x)&=&\frac{1}{\pi h^2}\int_{B(0,h)}[f(x+l\mathbf{n}_x+\xi)-f(x+l\mathbf{n}_x)]d\xi\nonumber\\
&=&\frac{1}{\pi h^2}\int_0^l\int_0^{2\pi}r[f(x+l\mathbf{n}_x+re^{i\theta})-f(x+l\mathbf{n}_x)]d\theta dr\nonumber\\
&&+\frac{1}{\pi h^2}\int_l^h\int_0^{2\pi}r[f(x+l\mathbf{n}_x+re^{i\theta})-f(x+l\mathbf{n}_x)]d\theta dr\label{eq3.2}
\end{eqnarray}
Since $f$ is harmonic in $D$, the first integral in \eqref{eq3.2} is zero due to the mean value property for $f$. Since $\partial D$ is analytic, there exist a conformal map $\Psi$ from $\mathbb{D}$ to $D$ and an $\epsilon>0$ such that $\Psi$ can be extended to a conformal map of $(1+\epsilon)\mathbb{D}$. This implies that the arc $\partial D\cap B(x,h)$ can be approximated by the tangent line segment at $x$ with fixed length of $O(h)$. More precisely, there exist constants $C_1>0$ and $C_2>0$ such that
$$\partial D\cap B(x,h)\subset \mathbf{t}_x\cdot[-C_1h,C_1h]+\mathbf{n}_x\cdot[-C_2h^2,C_2h^2]$$
uniformly for all $x\in\partial D$, where $\mathbf{t}_x$ is the unit tangent vector at $x$.
Therefore \eqref{eq3.2} gives
\begin{eqnarray*}
&&\Delta_h f(x+l\mathbf{n}_x)=\frac{1}{\pi h^2}\int_l^h\int_{-\pi/2-\arcsin(l/r)}^{\pi/2+\arcsin(l/r)}r[f(x+l\mathbf{n}_x+r\mathbf{n}_xe^{i\theta})-f(x+l\mathbf{n}_x)]d\theta dr\\
&&+\frac{1}{\pi h^2}\int_l^h\int_{-\arccos(l/r)}^{\arccos(l/r)}r[f(x+l\mathbf{n}_x-r\mathbf{n}_xe^{i\theta})-f(x+l\mathbf{n}_x)]d\theta dr+O(h^3)\\
&&:=I_1(x,l)+I_2(x,l)+O(h^3).
\end{eqnarray*}
Noticing that $f\in C^2(\bar{D})$, by Taylor expansion of $f$ about $x$ with coordinate directions $\mathbf{n}_x$ and $-i\mathbf{n}_x$, we get
\begin{eqnarray*}
I_1(x,l)&=&\frac{1}{\pi h^2}\int_l^h\int_{-\pi/2-\arcsin(l/r)}^{\pi/2+\arcsin(l/r)}r[\frac{\partial f(x)}{\partial \mathbf{n}_x}(r\cos\theta)]d\theta dr+O(h^2)\\
\end{eqnarray*}
where we used the fact the coefficient for $\frac{\partial f(x)}{\partial (i\mathbf{n}_x)}$  (which is $r\sin\theta$) is an odd function of $\theta$.

Similarly, by the definition of $f$ in $D_3$ and the Taylor expansion, we have
\begin{eqnarray*}
I_2(x,l)&=&\frac{1}{\pi h^2}\int_l^h\int_{-\arccos(l/r)}^{\arccos(l/r)}r[f(x-i\mathbf{n}_xr\sin(\theta))-f(x+l\mathbf{n}_x)]d\theta dr+O(h^2)\\
&=&\frac{1}{\pi h^2}\int_l^h\int_{-\arccos(l/r)}^{\arccos(l/r)}r[-l\frac{\partial f(x)}{\partial \mathbf{n}_x}]d\theta dr+O(h^2).
\end{eqnarray*}
The lemma follows by simple computations of $I_1(x,l)$ and $I_2(x,l)$.
\end{proof}

\subsection{Proof of Proposition \ref{propone}}
Now we are ready to prove Proposition \ref{propone}.

\begin{proof}[Proof of Proposition \ref{propone}]
Recall that $\int g(\xi) \omega_h(0,|d\xi|;D)=f_h(0)$ and
$\int g(\xi)\omega(0,|d\xi|;D)=f(0)$.
By Lemmas \ref{lem3} and \ref{lem6}, we have
\begin{eqnarray*}
&&f_h(0)-f(0)=\int_{D_2} G_h(0,z)\Delta_h f(z) dz\\
&=&(1+O(h))\int_{\partial D}\int_0^h G_h(0,x+l\mathbf{n}_x)\Delta_h f(x+l\mathbf{n}_x)dl|dx|.
\end{eqnarray*}

Using Lemma \ref{lem7}, we get
\begin{eqnarray*}
D(h)&:=&\int_{\partial D}\int_0^h G_h(0,x+l\mathbf{n}_x)\Delta_h f(x+l\mathbf{n}_x)dl|dx|\\
&=&\frac{1}{\pi h^2}\int_{\partial D}\frac{\partial f(x)}{\partial \mathbf{n}_x}\int_0^hG_h(0,x+l\mathbf{n}_x)\\
&&*[\frac{2}{3}h^2\sqrt{h^2-l^2}+\frac{l^2}{3}\sqrt{h^2-l^2}-lh^2\arccos(l/h)+O(h^4)]dl|dx|.
\end{eqnarray*}
Applying Proposition \ref{prop}, we obtain
\begin{eqnarray*}
D(h)&=&\frac{8}{\pi h^4}\int_{\partial D}\frac{\partial f(x)}{\partial \mathbf{n}_x}\int_0^h(G_D(0,x+l\mathbf{n}_x)+H_D(0,x)E^{li}[|\mbox{Im}(S_{T_{\mathbb{H}}})|]+o(h))\\
&&* [\frac{2}{3}h^2\sqrt{h^2-l^2}+\frac{l^2}{3}\sqrt{h^2-l^2}-lh^2\arccos(l/h)+O(h^4)]dl|dx|.
\end{eqnarray*}
Substituting the estimate in Lemma \ref{lem5} into the above equality, we see
\begin{eqnarray*}
D(h)&=&\frac{8}{\pi h^4}\int_{\partial D}\frac{\partial f(x)}{\partial \mathbf{n}_x}\int_0^h(lH_D(0,x)+H_D(0,x)E^{li}[|\mbox{Im}(S_{T_{\mathbb{H}}})|]+o(h))\\
&&* [\frac{2}{3}h^2\sqrt{h^2-l^2}+\frac{l^2}{3}\sqrt{h^2-l^2}-lh^2\arccos(l/h)+O(h^4)]dl|dx|.
\end{eqnarray*}
By the change of variables $l=h\cos\theta$, we have
\begin{eqnarray*}
D(h)&=&h[\frac{16}{45\pi}+\frac{8}{\pi}\int_0^{\pi/2}(\sin^2\theta-(\sin^4\theta)/3-\theta\cos\theta\sin\theta)E^{i\cos\theta}_{h=1}(|\mbox{Im}(S_{T_{\mathbb{H}}})|)d\theta]\\
&&*\int_{\partial D}\frac{\partial f(x)}{\partial \mathbf{n}_x}H_D(0,x)|dx|+o(h),
\end{eqnarray*}
which completes the proof.

\end{proof}

\section{Proof of density for limiting measure}\label{densityproof}
\begin{proposition}\label{proptwo}
Under the assumption of Proposition \ref{propone}, let $\psi$ be a conformal map from $D$ to $\mathbb{D}$ which sends $0$ to the origin. For $\theta\in[0,2\pi]$, define $m(\theta)=|\psi^{\prime}(\psi^{-1}(e^{i\theta}))|$. Then we have
\begin{equation*}
\int_{\partial D} H_D(0,z)\frac{\partial f(z)}{\partial \mathbf{n}_z}|dz|=\int_{0}^{2\pi}(g\circ\psi^{-1})(e^{i\phi})\rho(\phi)d\phi,
\end{equation*}
where
\begin{equation*}
\rho(\phi)=\frac{1}{4\pi^2}\int_0^{2\pi}\frac{m(\theta)-m(\phi)-m^{\prime}(\phi)\sin(\theta-\phi)}{1-\cos(\theta-\phi)}d\theta.
\end{equation*}
\end{proposition}

\begin{proof}
We define
\begin{equation*}
I(g)=\int_{\partial D} H_D(0,z)\frac{\partial f(z)}{\partial \mathbf{n}_z}|dz|.
\end{equation*}
Using the change of variables $\psi(z)=e^{i\theta}$ for $z\in\partial D$, we have
\begin{eqnarray*}
\frac{\partial f(z)}{\partial \mathbf{n}_z}&=&m(\theta)\frac{\partial (f\circ\psi^{-1})(e^{i\theta})}{\partial \mathbf{n}_{e^{i\theta}}}=m(\theta)\lim_{\epsilon\downarrow0}\frac{(f\circ\psi^{-1})(e^{i\theta}+\epsilon\mathbf{n}_{e^{i\theta}})-(f\circ\psi^{-1})(e^{i\theta})}{\epsilon}\\
&=&m(\theta)\lim_{\epsilon\downarrow0}\frac{(f\circ\psi^{-1})\left((1-\epsilon)e^{i\theta}\right)-(f\circ\psi^{-1})(e^{i\theta})}{\epsilon}.
\end{eqnarray*}
Under the same change of variables, the harmonic measure $H_D(0,z)|dz|$ transforms to $\frac{1}{2\pi}d\theta$. So
\begin{equation*}
I(g)=\frac{1}{2\pi}\int_{0}^{2\pi}m(\theta)\lim_{\epsilon\downarrow0}\frac{(f\circ\psi^{-1})\left((1-\epsilon)e^{i\theta}\right)-(f\circ\psi^{-1})(e^{i\theta})}{\epsilon}d\theta.
\end{equation*}
By the assumption of the proposition, $m(\theta)$ is a smooth function of $\theta$ and $f\circ\psi^{-1}\in C^2(\bar{\mathbb{D}})$. So the mean value theorem and the bounded convergence theorem give
$$I(g)=\frac{1}{2\pi}\lim_{\epsilon\downarrow0}\int_{0}^{2\pi}m(\theta)\frac{(f\circ\psi^{-1})\left((1-\epsilon)e^{i\theta}\right)-(f\circ\psi^{-1})(e^{i\theta})}{\epsilon}d\theta.$$
Note that $f\circ\psi^{-1}$ is the harmonic function on $\mathbb{D}$ with boundary data $g\circ\psi^{-1}$. We now have
$$I(g)=\frac{1}{2\pi}\lim_{\epsilon\downarrow0}\int_{0}^{2\pi}m(\theta)\frac{\int_0^{2\pi}H_{\mathbb{D}}\left((1-\epsilon)e^{i\theta},e^{i\phi}\right)[(g\circ\psi^{-1})(e^{i\phi})-(g\circ\psi^{-1})(e^{i\theta})]d\phi}{\epsilon}d\theta.$$
Fubini's theorem and interchange of $\theta$ and $\phi$ imply
$$I(g)=\lim_{\epsilon\downarrow0}\int_{0}^{2\pi}(g\circ\psi^{-1})(e^{i\phi})\int_0^{2\pi}\frac{m(\theta)H_{\mathbb{D}}\left((1-\epsilon)e^{i\theta},e^{i\phi}\right)-m(\phi)H_{\mathbb{D}}\left((1-\epsilon)e^{i\phi},e^{i\theta}\right)}{2\pi\epsilon}d\theta d\phi$$
Reflection symmetry implies
$$\int_0^{2\pi}m^{\prime}(\phi)\sin(\theta-\phi)H_{\mathbb{D}}\left((1-\epsilon)e^{i\phi},e^{i\theta}\right)d\theta=0.$$
So we can rewrite our integral as
\begin{eqnarray*}
I(g)&=&\frac{1}{2\pi}\lim_{\epsilon\downarrow0}\int_{0}^{2\pi}(g\circ\psi^{-1})(e^{i\phi})\int_0^{2\pi}[\frac{m(\theta)H_{\mathbb{D}}\left((1-\epsilon)e^{i\theta},e^{i\phi}\right)}{\epsilon}\\
&&-\frac{m(\phi)H_{\mathbb{D}}\left((1-\epsilon)e^{i\phi},e^{i\theta}\right)}{\epsilon}-\frac{m^{\prime}(\phi)\sin(\theta-\phi)H_{\mathbb{D}}\left((1-\epsilon)e^{i\phi},e^{i\theta}\right)}{\epsilon}]d\theta d\phi.
\end{eqnarray*}
Rotation and reflection symmetries imply
\begin{eqnarray*}
H_{\mathbb{D}}\left((1-\epsilon)e^{i\theta},e^{i\phi}\right)&=&H_{\mathbb{D}}\left((1-\epsilon)e^{i\phi},e^{i\theta}\right)=H_{\mathbb{D}}\left((1-\epsilon),e^{i(\theta-\phi)}\right)\\
&=&\frac{1}{2\pi}\frac{2\epsilon-\epsilon^2}{2-2\cos(\theta-\phi)-[2-2\cos(\theta-\phi)]\epsilon+\epsilon^2},
\end{eqnarray*}
where we have used the explicit expression for the Poisson kernel in $\mathbb{D}$ (see, e.g., Example 2.16 of \cite{Law05}).
So we now have
\begin{eqnarray*}
I(g)&=&\frac{1}{4\pi^2}\lim_{\epsilon\downarrow0}\int_{0}^{2\pi}(g\circ\psi^{-1})(e^{i\phi})\int_0^{2\pi}\frac{(2-\epsilon)[m(\theta)-m(\phi)-m^{\prime}(\phi)\sin(\theta-\phi)]}{2-2\cos(\theta-\phi)-[2-2\cos(\theta-\phi)]\epsilon+\epsilon^2}d\theta d\phi.
\end{eqnarray*}
Viewing $2-2\cos(\theta-\phi)-[2-2\cos(\theta-\phi)]\epsilon+\epsilon^2$ as a quadratic function of $\epsilon$, it is easy to show for any $0<\epsilon<1$ one has
$$2-2\cos(\theta-\phi)-[2-2\cos(\theta-\phi)]\epsilon+\epsilon^2\geq [1-\cos(\theta-\phi)]/2.$$
Therefore
\begin{equation}\label{eqlh}
|\frac{(2-\epsilon)[m(\theta)-m(\phi)-m^{\prime}(\phi)\sin(\theta-\phi)]}{2-2\cos(\theta-\phi)-[2-2\cos(\theta-\phi)]\epsilon+\epsilon^2}|\leq\frac{4|m(\theta)-m(\phi)-m^{\prime}(\phi)\sin(\theta-\phi)|}{1-\cos(\theta-\phi)}.
\end{equation}
L'H\^{o}pital's rule applied to the left hand of side \eqref{eqlh} implies that it is a bounded function of $\theta\in\mathbb{R}$ and $\phi\in\mathbb{R}$ (using the periodicity). So by the bounded convergence theorem (noting that $g\circ\psi^{-1}$ is also bounded) we have
$$I(g)=\frac{1}{4\pi^2}\int_{0}^{2\pi}(g\circ\psi^{-1})(e^{i\phi})\int_0^{2\pi}\frac{m(\theta)-m(\phi)-m^{\prime}(\phi)\sin(\theta-\phi)}{1-\cos(\theta-\phi)}d\theta d\phi,$$
which is the desired result.
\end{proof}
Let us remark that Theorem \ref{thm1} follows from Propositions \ref{propone} and \ref{proptwo}, and the change of variables $z=\psi^{-1}(e^{i\phi})$.

\section*{Acknowledgment}
We would like to thank the anonymous referee for many valuable comments and suggestions.
The research of T. Kennedy was supported in part
by NSF grant DMS-1500850.

\section{Appendix}
In this appendix, we prove the asymptotics for the potential kernel described in Lemma \ref{lempot}. We follow the methods introduced in Section 12 of \cite{Spi76} and \cite{FU96}. Let $\phi(\theta)$ be the characteristic function of the continuous-state random walk with $h=1$, i.e.,
$$\phi(\theta)=E e^{i X\cdot \theta}$$
where $\theta=(\theta_1,\theta_2)$ and $X=(X^{(1)},X^{(2)})$ is uniformly distributed in the disk of radius $1$.
\begin{lemma}\label{lema1}
\begin{eqnarray*}
\phi(\theta)=1-\frac{|\theta|^2}{8}+\frac{|\theta|^4}{192}+O(|\theta|^6), \theta\rightarrow 0\\
|\phi(\theta)|\leq\frac{4}{\pi}\min\{|\theta_1|^{-1},|\theta_2|^{-1}\}, \theta\rightarrow \infty.
\end{eqnarray*}
\end{lemma}
\begin{proof}
The first estimate in the lemma follows by Taylor expansion, while the second follows since
\begin{eqnarray*}
|\phi(\theta)|&=&|\frac{1}{\pi}\int_{-1}^{1}\int_{-\sqrt{1-x_1^2}}^{\sqrt{1-x_1^2}}\cos(x_1\theta_1)\cos(x_2\theta_2)dx_1dx_2|\leq\frac{4}{\pi|\theta_2|}
\end{eqnarray*}
and the symmetry of $\theta_1$ and $\theta_2$.
\end{proof}

The following lemma says our potential kernel is well-defined.
\begin{lemma}\label{lema2}
\begin{eqnarray*}
a(x)&=&\lim_{n\rightarrow\infty}\sum_{k=1}^n[p(k,0,0)-p(k,0,x)]\\
&=&\sum_{k=1}^2[p(k,0,0)-p(k,0,x)]+\frac{1}{(2\pi)^2}\int_{\mathbb{R}^2}\frac{1-e^{i\theta\cdot x}}{1-\phi(\theta)}\phi^3(\theta)d\theta
\end{eqnarray*}
\end{lemma}
\begin{proof}
By applying the continuous inversion formula, the proof is similar to the proof of P1 in Section 12 of \cite{Spi76} if one can show $\frac{1-e^{i\theta\cdot x}}{1-\phi(\theta)}\phi^3(\theta)\in L^1(\mathbb{R}^2)$. The latter is true because of Lemma \ref{lema1}.
\end{proof}

Let $Q(\theta)=E(X\cdot\theta)^2=\frac{|\theta|^2}{4}$ and $\psi(\theta)=1/(1-\phi(\theta))-2/Q(\theta)$. Then Lemma \ref{lema1} implies
\begin{equation}\label{eqa1}
\psi(\theta)=1/3+O(|\theta|^2)\mbox{ as } \theta\rightarrow0;
|\psi(\theta)|<2\mbox{ as } \theta\rightarrow\infty.
\end{equation}

Now we have all ingredients to prove Lemma \ref{lempot}.
\begin{proof}[Proof of Lemma \ref{lempot}]
By Lemma \ref{lema2} and the evenness of $\phi$, we see that
\begin{eqnarray}\label{eqa2}
a(x)&=&\sum_{k=1}^2[p(k,0,0)-p(k,0,x)]+\frac{2}{\pi^2}\int_{\mathbb{R}^2}\frac{1-\cos(x\cdot\theta)}{|\theta|^2}\phi^3(\theta)d\theta\nonumber\\
&&+\frac{1}{(2\pi)^2}\int_{\mathbb{R}^2}(1-e^{i\theta\cdot x})\psi(\theta)\phi^3(\theta)d\theta.
\end{eqnarray}
By the estimates in Lemma \ref{lema1} and \eqref{eqa1}, and the Riemann-Lebesgue lemma,
$$\frac{1}{(2\pi)^2}\int_{\mathbb{R}^2}(1-e^{i\theta\cdot x})\psi(\theta)\phi^3(\theta)d\theta\rightarrow\frac{1}{(2\pi)^2}\int_{\mathbb{R}^2}\psi(\theta)\phi^3(\theta)d\theta \mbox{ as }|x|\rightarrow\infty$$
which is a constant contributing to $C_0$ in the Lemma.

This gives the first $o(1)$ term
\begin{equation}\label{eqa3}
-\frac{1}{(2\pi)^2}\int_{\mathbb{R}^2}e^{i\theta\cdot x}\psi(\theta)\phi^3(\theta)d\theta=-\frac{1}{(2\pi)^2}\int_{\mathbb{R}^2}\cos(x\cdot \theta)\psi(\theta)\phi^3(\theta)d\theta.
\end{equation}

Let $B:=B(0,\pi):=\{z:|z|<\pi\}$ and $B^c=\mathbb{R}^2\setminus B$. Then the first integral together with the attached multiplicative term in \eqref{eqa2} can be written as the sum of the following two integrals
\begin{eqnarray}
I_1(x):=\frac{2}{\pi^2}\int_B\frac{1-\cos(x\cdot\theta)}{|\theta|^2}\phi^3(\theta)d\theta\label{eqa4}\\
I_2(x):=\frac{2}{\pi^2}\int_{B^c}\frac{1-\cos(x\cdot\theta)}{|\theta|^2}\phi^3(\theta)d\theta.\label{eqa5}
\end{eqnarray}
By the estimate in Lemma \ref{lema1} and the Riemann-Lebesgue lemma we have
$$I_2(x)\rightarrow\frac{2}{\pi^2}\int_{B^c}\frac{\phi^3(\theta)}{|\theta|^2} \mbox{ as } x\rightarrow\infty,$$
which leaves the second $o(1)$ term
\begin{equation}\label{eqa6}
-\frac{2}{\pi^2}\int_{B^c}\frac{\cos(x\cdot\theta)}{|\theta^2|}\phi^3(\theta)d\theta.
\end{equation}
We rewrite $I_1(x)$ as follows
\begin{equation}\label{eqa61}
I_1(x)=\frac{2}{\pi^2}\int_B\frac{1-\cos(x\cdot\theta)}{|\theta|^2}d\theta+\frac{2}{\pi^2}\int_B\frac{\phi^3(\theta)-1}{|\theta|^2}d\theta+\frac{2}{\pi^2}\int_B\frac{\cos(x\cdot\theta)}{|\theta|^2}(1-\phi^3(\theta))d\theta.
\end{equation}
By Lemma \ref{lema1}, the second integral in \eqref{eqa61} is a constant contributing to $C_0$ in the lemma, and by the Riemann-Lebesgue lemma the last integral in \eqref{eqa61} gives the third $o(1)$ term
\begin{equation}\label{eqa7}
\frac{2}{\pi^2}\int_B\frac{\cos(x\cdot\theta)}{|\theta|^2}(1-\phi^3(\theta))d\theta,
\end{equation}
and the first integral together with the attached multiplicative term in \eqref{eqa61} is equal to (using the proof of P3 in Section 12 of \cite{Spi76})
\begin{eqnarray}
\frac{8}{\pi^2}\int_0^{\pi/2}[\gamma+\ln\pi+\ln|x|+\ln(\sin\alpha)+\int_{\pi|x|\sin\alpha}^{\infty}\frac{\cos u}{u}du]d\alpha\label{eqa8}
\end{eqnarray}
where $\gamma$ is the Euler's constant.

It is clear that $\gamma+\ln\pi+\ln|x|+\ln(\sin\alpha)$ in \eqref{eqa8} as a function of $\alpha$ is integrable from $0$ to $\pi/2$, so the fourth $o(1)$ term is
\begin{equation}\label{eqa9}
\frac{8}{\pi^2}\int_0^{\pi/2}\int_{\pi|x|\sin\alpha}^{\infty}\frac{\cos u}{u}dud\alpha=\frac{2}{\pi^2}\int_{B^c}\frac{\cos(x\cdot\theta)}{|\theta|^2}d\theta.
\end{equation}
where the equality follows by reversing the procedure which led to \eqref{eqa8}.

Adding the four $o(1)$ terms, i.e., \eqref{eqa3}+\eqref{eqa6}+\eqref{eqa7}+\eqref{eqa9}, we get
\begin{eqnarray}
&&-\frac{1}{(2\pi)^2}\int_{\mathbb{R}^2}\cos(x\cdot \theta)\psi(\theta)\phi^3(\theta)d\theta+\frac{2}{\pi^2}\int_{\mathbb{R}^2}\frac{\cos(x\cdot\theta)}{|\theta|^2}(1-\phi^3(\theta))d\theta\nonumber\\
&=&\frac{1}{4\pi^2}\int_{\mathbb{R}^2}\cos(x\cdot \theta)[\frac{8}{|\theta|^2}-\frac{\phi^3(\theta)}{1-\phi(\theta)}]d\theta.\label{eqa10}
\end{eqnarray}
Noting that $\cos(x\cdot\theta)=\nabla\cdot \mathbf{b}(\theta)$ where $\mathbf{b}(\theta)=\sin(x\cdot\theta)(x_1/|x|^2,x_2/|x|^2)$, the divergence theorem gives
\begin{eqnarray}
&&\int_{\mathbb{R}^2}\cos(x\cdot \theta)[\frac{8}{|\theta|^2}-\frac{\phi^3(\theta)}{1-\phi(\theta)}]d\theta\nonumber\\
&=&\lim_{N\rightarrow\infty}\int_{B(0,N)}[\frac{8}{|\theta|^2}-\frac{\phi^3(\theta)}{1-\phi(\theta)}]\nabla\cdot \mathbf{b}(\theta)d\theta\nonumber\\
&=&\lim_{N\rightarrow\infty}|x|^{-1}\left(-\int_{B(0,N)}(\frac{x_1}{|x|},\frac{x_2}{|x|})\cdot\nabla[\frac{8}{|\theta|^2}-\frac{\phi^3(\theta)}{1-\phi(\theta)}]\sin(x\cdot\theta)d\theta\right).\label{eqa11}
\end{eqnarray}
We can apply the divergence theorem again to \eqref{eqa11}.
As a result, we see that \eqref{eqa10} has order $O(|x|^{-2})$.

Therefore, the proof of Lemma \ref{lempot} is complete if one can show $\Delta_{h} a(x/h)=\frac{1}{\pi}I_{B(0,h)}(x)$. But the latter is easy to verify (note that $a(x/h)=h^2a_h(x)$).

\end{proof}

\end{document}